\documentclass[11, english]{article}
\usepackage[margin=1in]{geometry}
\usepackage{mathrsfs}

\usepackage[dvipsnames]{xcolor}
\usepackage[square,sort,comma,numbers]{natbib}
\usepackage{amsmath}
\usepackage{amsthm}
\usepackage{amssymb}
\usepackage{setspace}
\usepackage{mathtools}
\usepackage{graphicx}
\graphicspath{ {./images/} }
\usepackage[hidelinks]{hyperref}
\usepackage{cleveref}
\usepackage{enumitem}
\usepackage{framed}
\usepackage{subcaption}
\usepackage{xspace}
\usepackage{thmtools} 
\usepackage{thm-restate}
\usepackage{thmtools}
\usepackage{thm-restate}

\usepackage{floatrow}

\usepackage{tikz}
\usepackage{mathdots}
\usepackage{xcolor}
\usepackage{diagbox}
\usepackage{colortbl}
\usepackage[absolute,overlay]{textpos}

\usetikzlibrary{shapes.misc}
\usetikzlibrary{decorations.pathmorphing}

\tikzset{snake it/.style={decorate, decoration=snake}}
\usetikzlibrary{math}

\usetikzlibrary{calc}
\usetikzlibrary{decorations.pathreplacing}
\usetikzlibrary{positioning,patterns}
\usetikzlibrary{arrows,shapes,positioning}
\usetikzlibrary{decorations.markings}

\tikzstyle{edge}=[very thick]
\definecolor{bostonuniversityred}{rgb}{0.8, 0.0, 0.0}
\definecolor{arsenic}{rgb}{0.23, 0.27, 0.29}
\tikzstyle{diredge}=[postaction={decorate,decoration={markings,
		mark=at position .95 with {\arrow[scale = 1]{stealth};}}}]
\tikzset{
    arrow/.style={decoration={markings, mark=at position 0.7 with
    {\fill(-0.09*#1,-0.03*#1) -- (0,0) -- (-0.09*#1,0.03*#1) -- cycle;}}, postaction={decorate}},
    arrow/.default=1
}
\tikzset{
    arow/.style={decoration={markings, mark=at position 1 with
    {\fill(-0.09*#1,-0.03*#1) -- (0,0) -- (-0.09*#1,0.03*#1) -- cycle;}}, postaction={decorate}},
    arow/.default=1
}
\tikzset{
    arrrow/.style={decoration={markings, mark=at position 0.9 with
    {\fill(-0.09*#1,-0.03*#1) -- (0,0) -- (-0.09*#1,0.03*#1) -- cycle;}}, postaction={decorate}},
    arow/.default=1
}

\newcommand{\fitellipsis}[2] 
{\draw [fill=white]let \p1=(#1), \p2=(#2), \n1={atan2(\y2-\y1,\x2-\x1)}, \n2={veclen(\y2-\y1,\x2-\x1)}
    in ($ (\p1)!0.5!(\p2) $) ellipse [ x radius=\n2/2+0cm, y radius=1.1cm, rotate=\n1];
}
\newcommand{\Fitellipsis}[2] 
{\draw [fill=white]let \p1=(#1), \p2=(#2), \n1={atan2(\y2-\y1,\x2-\x1)}, \n2={veclen(\y2-\y1,\x2-\x1)}
    in ($ (\p1)!0.5!(\p2) $) ellipse [ x radius=\n2/2+0cm, y radius=1.4cm, rotate=\n1];
}

\theoremstyle{plain}

\newtheorem*{thm*}{Theorem}
\newtheorem{thm}{Theorem}[section]
\Crefname{thm}{Theorem}{Theorems}

\newtheorem*{lem*}{Lemma}
\newtheorem{lem}[thm]{Lemma}
\Crefname{lem}{Lemma}{Lemmas}

\newtheorem*{claim*}{Claim}
\newtheorem{claim}{Claim}[section]
\crefname{claim}{Claim}{Claims}
\Crefname{claim}{Claim}{Claims}

\newtheorem{prop}[thm]{Proposition}
\Crefname{prop}{Proposition}{Propositions}

\Crefname{remar}{Remark}{Remarks}

\crefname{cor}{Corollary}{Corollaries}

\newtheorem*{conj*}{Conjecture}
\newtheorem{conj}[thm]{Conjecture}
\crefname{conj}{Conjecture}{Conjectures}

\Crefname{qn}{Question}{Questions}

\Crefname{obs}{Observation}{Observations}

\Crefname{ex}{Example}{Examples}

\theoremstyle{definition}

\Crefname{prob}{Problem}{Problems}

\newtheorem{defn}[thm]{Definition}
\Crefname{defn}{Definition}{Definitions}

\theoremstyle{remark}
\newtheorem*{rem}{Remark}

\newcommand{\remove}[1]{}

\title{\vspace{-1 cm}
On Independent Spanning Trees in Random and Pseudorandom Graphs}

\crefname{enumi}{property}{parts}

\newcommand{\whp}{\textbf{whp}}

\author{%
  Nemanja Dragani\'c\thanks{Mathematical Institute, University of Oxford, UK. 
    Research supported by SNSF project 217926. 
    Email: \texttt{nemanja.draganic@maths.ox.ac.uk}}%
  \and
  Keith Frankston\thanks{Center for Communications Research – Princeton, USA.  
    Email: \texttt{k.frankston@fastmail.com}}%
  \and
  Michael Krivelevich\thanks{School of Mathematical Sciences, Tel Aviv University, Tel Aviv 6997801, Israel. Research supported in part by NSF-BSF grant 2023688. Email: \texttt{krivelev@tauex.tau.ac.il}}%
  \and
  Alexey Pokrovskiy\thanks{Department of Mathematics, University College London, UK.  
    Email: \texttt{dralexeypokrovskiy@gmail.com}}%
  \and
  Liana Yepremyan\thanks{Department of Mathematics, Emory University, Atlanta, USA.  
    Email: \texttt{liana.yepremyan@emory.edu}.  Research is supported by the National Science Foundation grant 2247013: Forbidden and Colored Subgraphs.}%
}

\begin{document} 

\maketitle
\begin{abstract}
In 1989, Zehavi and Itai conjectured that every $k$-connected graph contains $k$ independent spanning trees rooted at any prescribed vertex $r$. That is, for each vertex $v$, the unique $r$–$v$ paths within these $k$ spanning trees are internally disjoint. This fundamental problem has received much attention, in part motivated by its applications to network reliability, but despite that has only been resolved for $k \le 4$ and certain restricted graph families.

We establish the conjecture for almost all graphs of essentially any relevant density. Specifically, we prove that there exists a constant $C > 1$ such that, with high probability, the random graph $G(n,p)$ contains $\delta(G)$ independent spanning trees rooted at any vertex whenever $C \log n/n \leq p < 0.99$. Since the lower bound on $p$ coincides (up to the constant $C$) with the connectivity threshold of $G(n,p)$, this result is essentially optimal. In addition, we show that $(n,d,\lambda)$-graphs with fairly mild bounds on the spectral ratio $d/\lambda$ contain $(1-o(1))d$ independent spanning trees rooted at each vertex, thereby settling the conjecture asymptotically for random $d$-regular graphs as well.
\end{abstract}

\section{Introduction}



Spanning trees $T_1,\dots,T_k$ (with $k\ge2$) in a graph $G$, rooted at a vertex $r$, are said to be \emph{ISTs}, short for \emph{independent spanning trees}, if, for every vertex $v\in V(G)$, the unique $r$–$v$ paths in the $T_i$'s are internally vertex-disjoint; that is, they share no vertices other than $r$ and~$v$.  These are also commonly known as vertex- or node-independent spanning trees\footnote{Edge-independent trees are also widely studied in the literature and are related to edge-connectivity instead (see~\cite{cheng2023independent} for more details)}. Independent spanning trees surface in an impressively wide range of settings: in computational biology, hypercube ISTs help to detect anomalies in mitochondrial DNA~\cite{silva2016}; in high-performance computing, ISTs provide fault-tolerant broadcast and low-diameter routes for multidimensional torus interconnections~\cite{tang2010}; carrier-grade IP networks leverage three edge-ISTs for fast rerouting under dual-link failures~\cite{gopalan2016}; and emerging server-centric datacentre fabrics employ so-called completely independent trees to preserve high throughput despite switch outages~\cite{li2021completely}.  We refer the interested reader to an extensive survey~\cite{cheng2023independent} by Cheng, Wang and Fan on the broad applications of ISTs in various fields ranging from biology to large-scale computing.

Probably the most famous conjecture in the area was made by Zehavi and Itai in the late '80s~\cite{zehavi1989three}, essentially at the conception of the notion of ISTs. It is widely referred to as the \emph{Independent Spanning Trees Conjecture}. To state it, we recall that an $n$-vertex graph $G$ is $k$-connected if deleting any $k-1<n$ vertices does not disconnect the graph; the largest such $k$ is the (vertex-) connectivity of $G$, denoted by $\kappa (G)$. 

\begin{conj}[Zehavi-Itai, '89]\label{conj:main} Every graph $G$ contains $\kappa(G)$ many ISTs rooted at $r$, for every choice of $r\in V(G)$.\footnote{For convenience, we will sometimes simply say the Zehavi-Itai conjecture holds when we mean $G$ has this property.}
\end{conj}
This conjecture is also of theoretical importance due to being a qualitative strengthening of Menger's Theorem~\cite{Menger}. Indeed, given any pair of vertices $u,v$, if we have $k$ independent spanning trees rooted at $u$, then the paths in these trees from $u$ to $v$ give $k$ internally vertex-disjoint paths as guaranteed by Menger's Theorem.

This problem received a lot of attention, and has been studied in numerous papers and for various classes of graphs. Before the conjecture was formally stated in 1989, Itai and Rodeh~\cite{itai1988multi} showed that the conjecture is true for $k=2$. Soon after, Cheriyan and Maheshwari~\cite{cheriyan1988} showed that the conjecture holds for $k=3$. Zehavi and Itai~\cite{zehavi1989three} independently arrived at the same result, and stated Conjecture~\ref{conj:main}~\cite{zehavi1989three} in the same work.  Much later, Curran, Lee, and Yu~\cite{Curran2006} settled the $k=4$ case.
Further results have been established for graphs with particular underlying topology, such as planar graphs~\cite{miura1999} or cube-like graphs~\cite{yang2007hypercube,werapun2012hypercube}.


Regarding general bounds, Censor-Hillel, Ghaffari, Giakkoupis, Haeupler and Kuhn~\cite{censor2017tight} showed that $k$-connected $n$-vertex graphs contain $\Omega(k/\log ^2 n)$ many ISTs via connected dominating sets (CDSs). 
Observe that one can naturally construct $k$ ISTs from $k$ disjoint connected dominating sets by selecting a spanning tree within each dominating set and attaching the remaining vertices of the graph as leaves. However, the CDS-based approach has inherent limitations: Censor-Hillel, Ghaffari, and Kuhn~\cite{HGK} proved that there exist $k$-vertex-connected $n$-vertex graphs containing only $O(k/\log n)$ many disjoint CDSs, so the best bound obtainable by this approach falls short of the conjectured bound by a logarithmic factor in $n$.

Nevertheless, for $d$-regular pseudorandom graphs, Draganić and Krivelevich~\cite{draganic2025disjoint} employed CDSs to show that such graphs contain $(1-o(1))d/\log d$ ISTs, thereby removing the dependence on $n$ for this class. In particular, this implies that random $d$-regular graphs contain $(1+o(1))d/\log d$ disjoint CDSs; their method yields the same bound for binomial random graphs $G(n,d/n)$ when $d\gg \log n$. These bounds are asymptotically best possible via the CDS approach, since in $G(n,d/n)$ the smallest dominating set typically has size $(1+o(1))n\log d/d$, thus there cannot be more than $(1+o(1))d/\log d$ disjoint CDSs. They further conjectured that random graphs above the connectivity threshold satisfy, at least approximately, the IST Conjecture. Our first main result establishes this conjecture and essentially shows that the Zehavi-Itai Conjecture typically holds for random graphs above the connectivity threshold:


\begin{thm}\label{thm:main}
    There exists $C>1$ such that, for any $C\log n/n \leq p \leq 0.99$, if $G\sim G(n,p),$  then, \whp{}, for every vertex $r$, there are $\delta(G)$ many ISTs rooted at $r$.
\end{thm}

\noindent As a classical result of Bollob\'as and Thomasson~\cite{bollobas1985random} shows that typically $\kappa(G)=\delta(G)$ for any $p=p(n)$, this implies that the Zehavi-Itai conjecture holds for binomial random graphs in the relevant range. 
 We note that we prove our result only for $p\leq 0.99$ for simplicity of presentation, and a slightly more careful argument would extend this to $p=1-o(1)$. For $p\ll \log n/ n$ the random graph $G(n,p)$ is typically not connected, so the Zehavi-Itai conjecture holds trivially. 
 
Very recently, independent from our work, an approximate version of Theorem~\ref{thm:main} was established by Hollom, Lichev, Mond, Portier, and Wang~\cite{HLMPW}, where, instead of finding  $\delta(G)$ many ISTs, they showed that for every $\varepsilon>0$ there exists $C=C(\varepsilon)$ such that, if $p\geq C\log{n}/n$ and $G\sim G(n,p),$ then, \whp{}, for every vertex $r\in V(G)$, there exist $(1-\varepsilon)\delta(G)$ such trees rooted at $r$ in $G$.

We note that attaining the exact bound, as opposed to an asymptotic solution, requires a considerably more delicate argument. To give some intuition, in a collection of ISTs, every vertex not adjacent to the root must have at least $\delta=\delta(G)$ incident edges belonging to distinct trees. Consequently, for vertices whose degree is close to $\delta$ in $G(n,d/n)$, one must carefully plan how their edges are distributed among the trees. On the other hand, if one aims only to construct $(1-\varepsilon)d$ ISTs, these complications are far less severe, and our proof can be considerably shortened in this case.


Our second main result resolves the Zehavi-Itai conjecture asymptotically for pseudorandom graphs, in particular, for $(n,d,\lambda)$-graphs.

\begin{defn}
An \(n\)-vertex graph \(G\) whose adjacency matrix has eigenvalues \(\lambda_1 \geq \lambda_2 \geq \dots \geq \lambda_n\) is called an \((n,d,\lambda)\)-\emph{graph} if \(G\) is \(d\)-regular and \(\max\{\lambda_2, |\lambda_n|\} \leq \lambda\).
\end{defn}

For readers less familiar with this notion, we refer to~\Cref{sec:ndlambda}, in particular~\Cref{lem:EML} (the Expander Mixing Lemma), which provides a concrete pseudorandomness corollary sufficient for our purposes.
These graphs are known for their strong connectivity and expansion properties, making them well-suited for applications such as fault-tolerant network design~\cite{hoory2006expander}. Despite being defined deterministically, they closely mimic the behavior of random graphs in many key respects. The reader is invited to consult comprehensive survey~\cite{krivelevich2006pseudo} for a thorough discussion of properties of $(n,d,\lambda)$-graphs.

In~\Cref{sec:ndlambda}, we establish the following asymptotic result for $(n,d, \lambda)$-graphs.
\begin{restatable}{thm}{NDLambdaResult}\label{thm:ndlambda result}
    Fix $\varepsilon>0$, and let $G$ be an $(n,d,\lambda)$-graph with $d/\lambda=\omega( \log d)$ and $d=o(n/\log^2 n)$. Then $G$ contains $(1-\varepsilon)d$ vertex-independent spanning trees for each choice of root.
\end{restatable}

Random $d$-regular graphs are typically $(n,d,\lambda)$-graphs with spectral ratio $d/\lambda=\Theta(\sqrt{d})$ (see~\cite{sarid2023spectral}), hence our results imply that such random graphs with $d=o(n/\log^2 n)$ satisfy the Zehavi-Itai conjecture asymptotically as well. For $d=\omega(\log n)$,  $G_1\sim G(n,d)$ can be coupled with $G_2\sim G(n,p)$ with $p=(1-o(1))d/n$ such that whp $G_2\subseteq G_1$ (see~\cite{gao2020sandwiching} and its bibliography). Thus \Cref{thm:main} implies that the conjecture holds asymptotically for all larger $d$ as well.

\begin{thm}\label{thm:random regular}
    Fix $\varepsilon>0$. Let $G$ be a random $d$-regular graph for $d\geq d_0(\varepsilon)$. For every vertex $r \in V(G)$, there exist $(1 - \varepsilon)d$ edge-disjoint independent spanning trees rooted at $r$.
\end{thm}

Hollom et al.~\cite{HLMPW} also established the existence of $(1-\varepsilon)d$ ISTs for random regular graphs $G(n,d)$, but their result requires $d\gg \log n$, whereas we achieve it for $d=\Omega(1)$. In the same work, they further showed that $\lfloor d/4 \rfloor$ ISTs exist for all $d$, with the caveat that the root $r$ is not arbitrary but can be chosen from most of the vertices of $G$.





\section{Preliminaries}

\subsection{{Notation}}
Throughout this paper we adopt the usual graph‐theoretic conventions. For a graph \(G\), we write $V(G)$ for its vertex set and $E(G)$ for its edge set. For $S,B\subseteq V(G)$, $G[S]$ denotes the subgraph of $G$ induced by $S$, $N_G(S)$ denotes the \emph{external} neighbourhood of $S$, while $\Gamma_B(S)$ denotes all vertices in $B$ adjacent to at least one vertex in $S$. Whenever it is clear from the context, we drop the subscript. Given two subsets $S_1,S_2\subseteq V(G)$, $e_G(S_1,S_2)$ counts the edges with one endpoint in each. We write $d_G(v,S)$ to denote the number of neighbours of vertex $v$ in the set $S$. 
We write $\delta(G)$ and $\Delta(G)$ for the minimum and maximum degree of $G$, respectively. Let $G(n,p)$ denote the binomial random graph on $n$ vertices in which each of the $\binom n2$ possible edges appears independently with probability $p$.
We say that an event $\mathcal{E}$ holds \emph{with high probability} (abbreviated \whp{}) if $\Pr(\mathcal{E})\to1$ as $n\to\infty$. All logarithms are natural (base $e$) unless explicitly indicated otherwise.

\subsection{Properties of random graphs}

We first state Chernoff's inequality, before we proceed to collect some useful properties that hold for $G(n, p)$ \whp{}:

\begin{thm}[Chernoff bounds]
Let $X\sim B(n,p)$ be a binomial random variable with $n$ trials and success probability $p$. Then the following hold.
\begin{itemize}
    \item $P[X>(1+\delta)np ]\leq e^{-\delta^2np/(2+\delta)}$ for $\delta>0$.
    \item $P[X<(1-\delta)np]\leq e^{-\delta^2 np/2}$ for $0<\delta<1$.
\end{itemize}

\end{thm}

The next lemma states some typical properties of $G(n,p)$ that we will use in our proofs. Various expressions in the lemma are tailored for our needs and are not optimized.
\begin{lem}\label{lem: properties of gnp}
    Given sufficiently large $C > 1$, for any $C\log n/n \leq p \leq 0.99$, let $G\sim G(n,p)$. Denote by $S$ the set of vertices of degree less than $np-0.9\sqrt{2np\log n}$. Then, with the stated probabilities,
$G$ satisfies the following properties:
     \begin{enumerate}[label=(\alph*)]
    \item \label{itm:mindeg} \textbf{Min and max degree:} we have \whp{}
    $\delta(G)= np-(1\pm\varepsilon)\sqrt{2np(1-p)\log n}$, where $\varepsilon\rightarrow 0$ as $C\rightarrow \infty$; furthermore $\Delta(G)\leq 2pn$ with probability at least $1-o(n^{-10})$.
    
    \item \label{itm:lowdegset} \textbf{Low-degree set:}
    For $p=o(1)$, $|S|\leq n^{0.2}$ with probability $1-o(n^{-3})$. 
    
    \item \label{itm:sparseboundary} \textbf{Sparse boundary:}
    For $p\leq n^{-0.49}$, every vertex outside of $S$ has at most $\sqrt{np}$ neighbours in $S\cup N(S)$ with probability at least $1-o(n^{-3})$.
    
    \item \label{itm:commonneighbors} \textbf{Common neighbors:}
    For $p\leq n^{-0.49}$, every pair of vertices $u,v$ has at most $\sqrt{np}$ common neighbours with probability at least $1-o(n^{-3})$.
    \item\label{p:independence}\textbf{Independence of low-degree set:} For $p\leq n^{-0.49}$ the set $S$ is independent \whp{}.
    
    \item \label{itm:connectivity} \textbf{Connectivity:}
    $G$ is $\delta(G)$-connected \whp{}.
\end{enumerate}

\end{lem}
\begin{proof}\hfill
\begin{enumerate}[label=(\alph*)]
    \item See Chapter 3 in~\cite{bollobas01random}.
    \item 
    Consider a fixed set $S$ of size $|n^{0.2}|$.
    Apply Chernoff bounds to get that the probability that a vertex has degree at most $np-0.9\sqrt{2np(1-p)\log n}$ outside of $S$ is at most $e^{-0.81(1-p)\log n}=n^{-0.81(1-p)}$.
    Note that these events are independent for all vertices in $S$. Hence the probability that all vertices in $S$ have small degree outside is at most $(n^{-0.81(1-p)})^{|S|}$. Thus, by a union bound, we get that the probability of such a set existing is at most
    $$
    \binom{n}{n^{0.2}}\cdot n^{-0.809n^{0.2}}\leq (en^{0.8})^{n^{0.2}}n^{-0.809n^{0.2}}\leq n^{-n^{0.2}/1000}\leq n^{-4}
    $$
    Thus with probability $1-o(n^{-3})$ there is no set $S$ which has small neighbourhood outside of $S$, and in particular does not have a small neighbourhood in general.
    
    \item  Fix a vertex $v$. Expose all pairs in $G-v$. By \ref{itm:lowdegset}, we have that the set $S'$ of vertices of degree at most $np-0.9\sqrt{2np(1-p)\log n}$ in $G-v$ is at most $n^{0.21}$, and thus $|N(S')|\leq 2n^{1.21}p$ with probability at least $1-o(n^{-9})$, by the first part. Now expose the neighbours of $v$. The expected number of neighbours in $S'\cup N(S')$ is at most $pn^{1.21}\cdot p=n^{1.21}p^2=o(\sqrt{np})$. Thus $v$ has at most $o(\sqrt{np})$ neighbours in $S\cup N(S)$ with probability $1-o(n^{-3})$ in case $np=\omega(\log^2 n)$. Otherwise, the probability that $v$ has at least $100$ neighbours in $S'\cup N(S')$ is at most $|S'\cup N(S')|^{100}p^{100}\leq n^{-5}$. Hence we are done by a union bound over all vertices $v$ in both cases.
    \item  For a given pair of vertices $u, v$, the size of their common neighbourhood $N(u,v)$ is binomially distributed with parameters $(n-2, p^2)$. Therefore, the claim follows from Chernoff-type bounds for the binomial distribution and the union bound over all pairs of vertices.
    \item 
    Fix a pair $u,v\in V(G)$.
    The probability that $u$ (or $v$ respectively) has at most $np-0.9\sqrt{2np(1-p)\log n}$ neighbours in $V(G)-\{u,v\}$ is at most $n^{0.21}/n$ by \Cref{itm:lowdegset}. Thus the event that $uv$ is an edge in $S$ has probability at most $(n^{-0.79})^2p=o(n^{-2})$. By a union bound over all pairs of edges, there is no edge in $S$ \whp{}.
    
\end{enumerate}
    
\end{proof}


    A balanced random bipartite graph typically contains a perfect matching when above the connectivity threshold. Below is a proof with a bound on the success probability.
\begin{lem}\label{lem:matching}
    Let $G\sim G(n,n,p)$ be a binomial random bipartite graph with parts of size $n$, and edge probability $p\geq 40\log n/n$. Then $G$ has a perfect matching with probability at least $1-n^{-\frac{pn}{2\log n}+4}\geq   1-o(n^{-10})$.
\end{lem}
\begin{proof}
    Denote the parts by $A,B$. If $G[A,B]$ has no perfect matching, then  by Hall's theorem there is a set $A_0\subset A$ of cardinality $|A_0|\le n/2$ with $|N(A_0,B)|< |A_0|$, or a set $B_0\subset B $ of cardinality $|B_0|\le n/2$ with $|N(B_0,A)|< |B_0|$. The probability of this can be bounded from above by \[\sum_{k=1}^{n/2} \binom{n}{k}^2(1-p)^{k(n-k)}\le \sum_{k=1}^{n/2} \left(\frac{en}{k}\right)^{2k}e^{-pkn/2}\leq 
     \sum_{k=1}^{n/2} \left(3n\right)^{2k}n^{-pkn/(2\log n)}\leq n\left(3n\right)^{2}n^{-pn/(2\log n)},\]
which implies the required bound as $n\rightarrow\infty$.
\end{proof}

\begin{lem}\label{lem:existenceofroots}
    Let $G\sim G(n,p)$ where $p\geq 10^5\log n /n$, let $k \leq n^{0.51}/2$, and let $Q$ be a subset of $V(G)$ of cardinality $k$. Then with probability at least $1-o(n^{-1})$ there are $k$ many vertex disjoint paths each of length $n/50k$ such that every vertex of $Q$ is an endpoint of one of these paths. 
\end{lem}
\begin{proof}

   We expose the random edges of $G(n,p)$ in bunches. Start with $Q_0=Q$ and expose all the edges between $Q_0$ and $V(G)\setminus Q_0$, call the graph induced by these edges $G_1$. Notice that $|V(G)\setminus Q_0|\geq n- n^{0.51}/2$, thus by Lemma~\ref{lem:matching} (by adding dummy vertices to $Q_0$ to be of the same size as $V(G)\setminus Q_0$) with probability $1-o(n^{-2})$ we can find a matching $M_1$ saturating $Q_0$ and going into $V(G)\setminus Q_0$ in $G_1\subset G$. We now let $Q_1$ be the endpoints of $M_1$ disjoint from $Q_0$. Expose the edges between $Q_1$ and $V(G)\setminus (Q_0\cup Q_1)$, and call the graph induced by these edges $G_2\subset G$. Just like in the previous step, with probability $1-o(n^{-2})$  there exists a matching $M_2$ between $Q_1$ and $V(G)\setminus (Q_0\cup Q_1)$ in $G_2$. Repeat this process  $\ell=n/50k$  times. At the $j$-th step, we will still have $|V(G)\setminus \cup_{i=1}^{j-1}{Q_i}|\geq 49n/50$, thus again with probability $1-o(n^{-2})$ there is a matching $M_{j}$ saturating $Q_{j-1}$ and going into $V(G)\setminus \cup_{i=1}^{j-1}{Q_i}$ in the subgraph $G_{j}\subset G$. Taking a union bound over all $\ell$ rounds, we get that with probability $1-o(n^{-1})$ the union of these matchings  $\{M_i\}_{i=1}^{\ell}$ gives us the desired collection of paths.
    \end{proof}
We will use the following randomized version of the previous lemma:
\begin{lem}\label{lem:existenceofroots_random}
    Let $G\sim G(n,p)$ where $p\geq 10^6\log n /n$, let $k \leq n^{0.51}$, $\ell\le n/100k$. For any $Q\subseteq W\subseteq V(G)$ with $|Q|=k$, $|W|\ge n/2$, we can define a randomized set $P\subseteq V(G)$ with the following properties:
    \begin{itemize}
        \item With probability at least $1-o(n^{-1})$ there are $k$ many vertex disjoint paths each of length $\ell$ such that every vertex of $Q$ is an endpoint of one of these paths. 
        \item The distribution of $P\setminus Q$ is that of a uniformly chosen random subset of $W$ of order $k\ell$.
    \end{itemize}
\end{lem}
\begin{proof}
Let $\sigma$ be a  random permutation of $W$ chosen uniformly at random from all such permutations that fix $Q$. Note that $\sigma^{-1}$ also has the same distribution. Let $G'=\sigma(G[W])$, noting that for distinct $x,y\in W$, $xy$ is an edge of $G'$ independently with probability $p$ i.e. $G'$ has the distribution of another Erd\H{o}s-Renyi random graph on vertex set $W$. By Lemma~\ref{lem:existenceofroots}, with probability $1-o(n^{-1})$, $G'$ contains a set  $k$ many vertex disjoint paths each of length $n/100k$ such that every vertex of $Q$ is an endpoint of one of these paths. When this occurs, shorten these to have length $\ell$, and let   $P'$ be the vertex set of these paths noting $|P'|=k\ell$. In outcomes when this doesn't occur, let $P'$ be an arbitrary subset of $W$ of order $k\ell$.  Let $P=\sigma^{-1}(P')$. Since $\sigma^{-1}$ is a uniformly chosen random permutation fixing $Q$, every subset of $W$ of order $k\ell$ is equally likely to end up as $P$.
\end{proof}

\section{Proof of Theorem~\ref{thm:main}}

We start by defining the notion of a \emph{nice} collection of trees, which will be useful for building ISTs:

\begin{defn}\label{def:nice collection}
    Given a graph $G$, let $\mathscr S=\{S_i\}_{i\in [t]}$ be a collection of trees rooted at the same vertex $r$, but pairwise disjoint otherwise.
    For every $v\in V(G)$, let $I(v)=\{i: v\notin S_i\cup N(S_i)\}$ be the set of indices of trees which do not contain $v$ or any of its neighbours. The collection $\mathscr S$ is \emph{nice} if, for all $v\in V(G)$ we can, for each $i\in I(v)$, find $u_i\notin \bigcup_{j\in [t]} S_j$
     and $w_i\in S_i$ such that the $u_i$ are all distinct and $v$--$u_i$--$w_i$ is a two-edge path in $G$.
\end{defn}
The next result shows how to construct ISTs from a nice collection of trees.
\begin{lem}\label{lem:nice trees are sufficient} If a graph $G$ contains a nice collection of trees $\mathscr T=\{S_i\}_{i\in [t]}$ then it contains a collection of ISTs, $\{T_i\}_{i\in [t]}$, such that $S_i\subseteq T_i$ for each $i$.
\end{lem}
\begin{proof}

    For each tree $S_i$, add all vertices from $N(S_i)$ to $S_i$ by attaching each of them as a leaf to an arbitrary neighbour in $S_i$, to obtain tree $S_i'$.

    Next, for every vertex $v$ outside $S_i'$, add the edge $vu_i$ guaranteed by the definition of a nice collection of trees where $u_i\notin \bigcup_{i\in [t]} S_i$. Note that that by doing this, for each $i$ we produced a spanning tree $T_i$, as we were only adding leaves.
  To avoid any confusion, note that the particular edge \(u_i w_i\) from \Cref{def:nice collection} need not appear in \(T_i\); in the first step of the proof, \(u_i\) may instead connect to some other vertex in \(S_i\); the definition only guarantees that such an edge exists.
    Let us prove now that $T_1,\ldots, T_t$ is a vertex-independent collection.

    For this, fix two trees $T_i,T_j$, and let us prove that for every $v$ the paths between $r$ and $v$ in $T_i$ and $T_j$ are internally vertex disjoint. We have the following cases:
    \begin{itemize}
        \item $v\in S_i\cup N(S_i)$, and $v\in S_j\cup  N(S_j)$: In this case all internal vertices in the $r$--$v$ path in $T_i$ are contained in $S_i$, while the internal vertices in the $r$--$v$ path in $T_j$ are contained in $S_j$, completing this case.
        \item $v\in S_i\cup N(S_i)$, and $v\notin S_j\cup N(S_j)$: All internal vertices of the first path are in $S_i$, while the second path has all internal vertices in $S_j\cup\{u_j\}$ where $u_j\notin \bigcup_{i \in [t]} S_i\supset S_i$, proving this case.
        \item $v\notin S_i\cup N(S_i)\cup S_j\cup N(S_j)$: Apart from vertices $u_i,u_j$, the internal vertices of the two paths lie in $S_i$ and $S_j$ respectively. Since $u_i$ and $u_j$ are distinct by construction, this completes the proof.
    \end{itemize}
\end{proof}

\begin{proof}[Proof of Theorem~\ref{thm:main}]
We consider two cases depending on $p$.
\vspace{10pt}
\textbf{The dense regime:} $\log^2 n/\sqrt{n} \leq p\leq 0.99$. This is the simpler case, so it serves as a good warm-up. We prove that the following claim holds, which will imply the statement in this case. 

\begin{claim}\label{cl: nice trees in dense case}
    The following holds with probability $1-o(1/n)$ for every integer $k\in[np/2,np]$. For every ordered pair of non-adjacent vertices $(u,v)$ in $G$,
    let $K$ be the set of $u$'s first $\min\{k,d(u)\}$ neighbours (according to the natural ordering on $[n]$). There exists a matching between $K\setminus N(v)$ and $N(v)\setminus K$ which covers the smaller side.
\end{claim}

Before we show the proof of the claim, let us see how it implies the statement of the theorem in this case. Suppose $G$ is a graph for which the conclusion of the claim holds. Denote $\delta=\delta(G)$, and note that by \Cref{lem: properties of gnp} we have that \whp{} $\delta\in [np/2, np]$ with room to spare.

Let $u$ be an arbitrary vertex in $G$, fixed to be the root. Apply the claim with $k=\delta$, and let $K=\{v_1,\ldots,v_\delta\}$ be the first $\delta$ neighbours of $u$. Define $S_i=\{u,v_i\}$. Then the claim exactly shows that the $S_i$'s are a nice collection of trees, which completes the proof by \Cref{lem:nice trees are sufficient}. Indeed, to verify that \Cref{def:nice collection} applies to our tree collection, note that if $v\in N(u)$ then $I(v)=\emptyset$, as $v$ is in the neighbourhood of the root $u$. Otherwise, let $I(v)=\{i: v_i\not\in N(v)\}$. Note that since $|N(v)|\geq \delta=|K|$, we have $|K\setminus N(v)|\leq |N(v)\setminus K|$, so there is a matching in $G$ which covers $K\setminus N(v)$ while the other endpoints are in $N(v)\setminus K\subseteq V(G)\setminus\bigcup_{i\in[\delta]} S_i$. For each $v_i\in K\setminus N(v)$, let $u_i\in N(v)\setminus K$ be the vertex it is matched to. Noting that $K\setminus N(v)=\{v_i: i\in I(v)\}$,  we have now defined distinct vertices $u_i$ for all $i\in I(v)$ such that $v$--$u_i$--$v_i$ is a two-edge path from $v$ to $v_i\in S_i$.

\begin{proof}[Proof of \Cref{cl: nice trees in dense case}]
Fix $u$ and $v$, and expose the neighbourhood of $u$. If $v$ is adjacent to $u$ then there is nothing to prove. By assumption on $p$ and Chernoff bounds, we have that with probability at least $1-o(n^{-8})$ we have $d(u)\geq np/2$.
Let $K$ be the set of $u$'s first $\min\{k,d(u)\}\geq \sqrt{n}\log{n}$ neighbours. We now expose the neighbours of $v$;
by the upper bound on $p$, we know that $|N(v)\setminus K |\geq d(v)/10^4\geq \sqrt{n}\log n$  with probability at least $1-o(n^{-8})$. Similarly $|K\setminus N(v)| \geq |K|/10^4\geq \sqrt{n}\log n$. Now, expose all edges between these two disjoint sets $K\setminus N(v)$ and $N(v)\setminus K$. We claim that there is a matching saturating the smaller side with probability at least $1-o(n^{-4})$, which completes the proof by a union bound over all pairs $u,v$ and all of the at most $n$ choices for $k$.
 To see this, choose $A$ and $B$ to be subsets of $N(v)\setminus K$ and $K\setminus N(v)$ respectively of size $n_0=|A|=|B|=\min\{|N(v)\setminus K|, |K\setminus N(v)|\}$.
 Since $n_0\geq \sqrt{n} \log n$ and $p\geq \log^2 n/\sqrt{n}\geq 40 \log n_0/n_0$ we have by \Cref{lem:matching} that there is a required matching between $A$ and $B$ with probability at least $1-o(n_0^{-10})>1-o(n^{-5})$.
\end{proof}




\vspace{20pt}
\textbf{The sparse regime:} $C\log n/n \leq p \leq \log^2 n/\sqrt{n}$.

Let $p_2=\frac{10^6\log n}{ n}$, $p_3=p/1000$, and $p_1$ be defined by $1-p=(1-p_1)(1-p_2)(1-p_3)$, notice that $p_1\sim 0.999p$. Let $G_i\sim G(n,p_i)$, and observe that if $G\sim G(n,p)$ then  $G=G_1\cup G_2\cup G_3$.

\medskip 

\paragraph{\emph{Proof outline.}}
We first outline the proof for a fixed root $r\notin S\cup N_G(S)$; the full argument requires slightly more careful probability estimates, since in the end we perform a union bound over all vertices.
 We expose the edges of $G$ in stages, one graph  $G_i$ at a time for a chosen subset of pairs. First we expose the whole graph $G_1$, which typically contains 99.9\% of all edges. We then identify the set $S$ of vertices of small degrees in $G_1$. Since $G_1$ contains a big majority of the edges, we can say that after we eventually reveal all of  $G_2\cup G_3$, the vertices with smallest degrees in $G$ are likely to be in $S$. 
 Next we reveal the edges of $G_2\cup G_3$ touching $S$. Notice that all edges of $G$ touching $S$ have now been revealed. Denote $\delta=\min\{d_G(v)\mid v\in S\}$. As indicated before, we expect $\delta$ to be the minimum degree of $G$.

Next we pick an arbitrary vertex $r\notin S\cup N_G(S)$, and  expose all edges containing it in $G_2\cup G_3$. We will argue that \whp{} the $G$-degree of $r$ outside $S\cup N_G(S)$ is at least $\delta$, and we pick $\delta$ neighbours to be the initial edges of the required $\delta$ trees. Then we expose the edges of $G_2$ outside of $V(G)-(S\cup N_G(S))$. The density of $G_2$ is sufficient  to find \whp{} a family of $\delta$ many paths in $G_2$, denoted by $\mathcal{P}$, all rooted at $r$ and pairwise disjoint otherwise, and of length $\Omega(\log n/np^2)$. We hope to show that these paths (viewed as trees) are nice in $G$ \whp{}. Since the set $P$ of all vertices in these paths  can essentially be viewed as a random set of a given size, by standard concentration bounds, typically no vertex has a lot of its $G_1$-degree into $P$.
After that we expose the neighbours in $G_3$ of every vertex $v\in V- S$ towards $R=V(G)-S-P$. As $S\cup P$ is small, we can show that for all such $v$ \whp{} $d_{G_1\cup G_3 }(v,R)\geq \delta$. For every $v\in V(G)-S$, denote its neighbourhood in $R$ by $B_v$. Also, for every $v\in S$, denote $B_v:=N_G(v)$. Thus, \whp{} for every $v\in V(G)$, $B_v$ is a set of cardinality of size at least $\delta$ outside of $P$. Now, for each $v\in V(G)$ reveal the edges  of $G_3$ between $P$ and $B_v$. Let $\mathscr P$ be the set of all paths, and $H_v$ be the auxiliary bipartite graph between $\mathscr P$ and $B_v$, with an edge connecting $P_i$ to $u\in B_v$ if there is an edge between $P_i$ and $u$ in $G_3$. As the number of paths is $|\mathscr P|=\delta\leq |B_v|$, and the expected degree of each vertex in this auxiliary graph is at least $50\log n$, by a standard Hall-type argument, there is a matching covering $\mathscr P$ with probability $1-o(1/n)$. By a union bound, this holds for every vertex. This is exactly the setup for \Cref{lem:nice trees are sufficient}, which will complete the proof.

Now we fill in the details of the proof stage by stage.
\vspace{1cm}
\paragraph{\textbf{Exposing $G_1$ and identifying small degree vertices and their neighbours.}}
We first expose all edges in $G_1$. By \Cref{lem: properties of gnp}\ref{itm:lowdegset} we have that the set of vertices $S$ of degree at most $np_1-0.9\sqrt{2np_1 \log n}$ is of size at most $|S|\leq n^{0.2}$. 

Note that, by Chernoff bounds and a union bound, with probability $1-n e^{-4.1\log n/2}= 1-o(n^{-1})$ all vertices in $V(G)$ have degree at least $$n(p_2+p_3)-\sqrt{4.1n(p_2+p_3)\log n}$$ in $G_2\cup G_3$. Hence with the same probability vertices outside $S$ will have degree in $G$ at least 
\begin{align}\label{min degree in of large guys}   np_1-0.9\sqrt{2np_1\log n}+n(p_2+p_3)-\sqrt{4.1n(p_2+p_3)\log n}\geq np-0.95\sqrt{2np\log n},\end{align}
where we used $0.9\cdot\sqrt{0.999}+\sqrt{2.05\cdot0.001}<0.95.$ Note that here we just calculated the probability of the mentioned event, we did not yet expose the randomness of $G_2$ and $G_3$.

Denote $\delta=\min\{d_G(v)\mid v\in S\},$ anticipating that the lowest degree vertex will \whp{} be in $S$, as by \Cref{lem: properties of gnp} the minimum degree in $G$ will be $np-(1\pm o_C(1))\sqrt{2np\log n}$.

Now expose all edges in $G_2\cup G_3$  touching $S$. By \Cref{lem: properties of gnp}\ref{itm:mindeg} we can assume that the set $N_G(S)=N_{G_1\cup G_2\cup G_3}(S)$ has size at most $2np|S|$, and by \ref{p:independence} that $S$ is independent.

\paragraph{\textbf{Fixing a root and finding a potentially nice collection of trees.}}

So far we only exposed $G_1$, and the edges touching $S$ in $G$.
Fix an arbitrary root $r\in V(G)$.
Our aim is to show that with probability at least $1-o(n^{-1})$, the vertex $r$ is a root of a collection of $\delta$ many ISTs. Then we would be done by a union bound over all $n$ vertices. 

If $r$ is not in $S$, also expose all edges containing $r$ in $G_2\cup G_3$. In that case, since  by Lemma~\ref{lem: properties of gnp}\ref{itm:sparseboundary} with probability $1-o(n^{-1})$ the root $r$ has at most $\sqrt{np}$  neighbours in $S\cup N(S)$, by \Cref*{min degree in of large guys} with the same probability $r$ has least $\delta$ neighbours outside $S\cup N(S)$. Fix  a subset $Q$ of $\delta$ such  neighbours if $r\notin S$, and otherwise if $r\in S$, then let $Q$ be a set of $\delta$ arbitrary vertices in $N(r)$; note that in the latter case $N(r)\subseteq N(S)$ as $S$ is independent by \Cref{lem: properties of gnp}\ref{p:independence}. The edges from $r$ to $Q$ will be edges which belong to distinct trees in the nice collection we are about to find.

Next, we expose all edges of  $G_2$ not touching $S\cup N(S)$ or $r$. 
We define the set $W$ of vertices which we initially use to find our nice collection of trees: $W:=V(G)-(S\cup N(S)\cup \{r\})+Q$.

Apply \Cref{lem:existenceofroots_random} to $G= G(n, p_2), W, Q, k=\delta, \ell= \lceil\frac{10^5\log n}{np^2}\rceil$, to get a set $P$ such that with probability $1-o(n^{-1})$, there is  collection $\mathscr P$ of vertex-disjoint paths $P_1,\ldots,P_\delta$, each having exactly one endpoint in $Q$ (denote by $r_i$ the endpoint of $P_i$), of length $\lceil\frac{10^5\log n}{np^2}\rceil$, and having $V(\mathscr P)=P\cup Q$. 
Note that  $$|P|= \delta \Big\lceil\frac{10^5\log n}{np^2}\Big\rceil \leq\max\Big\{\frac{2\cdot10^5\log n}{p}, np\Big\}\leq 2\cdot 10^5n/C,$$ since $\delta \leq np$, where the second inequality considers the two cases when the expression under the ceiling is either less or more than $1$. 

From Lemma~\ref{lem:existenceofroots_random}, the set $P$ is a uniformly at random chosen set of size $\delta \lceil\frac{10^5\log n}{np^2}\rceil$ in $W-Q $. 
More formally, if $E$ is the event that in the random graph $G_2[W]$ the collection $\mathscr P$ exists, then we have that $Pr[P=S_1\mid E]=Pr[P=S_2\mid E]$ for every two sets $S_1,S_2$ of size $\delta \lceil\frac{10^5\log n}{np^2}\rceil$, as every vertex is equally likely to be included in $P$. 

Thus in the previously exposed graph $G_1$, with probability $1-o(n^{-1})$ every vertex $v\in V(G)$ has at most $\max\{10^6\log n,2np^2\} $ neighbours in $P-Q$. This is because the  number of neighbours of $v$ in $P$ is distributed hypergeometrically $\text{Hyp}(|W|-|Q|,K,m)$ where $K<1.1np$, and $m\leq \max\{\frac{2\cdot 10^5\log n}{p},np\}$; note that $|W|=(1-o(1))n$. Therefore, the expected number of such neighbours is at most $1.2np\cdot \max\{\frac{2\cdot 10^5\log n/p}{n}, p\}\leq \max\{3\cdot 10^5\log n, 1.2np^2\}$, so the claim follows by Chernoff-type bounds for the hypergeometric distribution (see Theorem 2.10 in \cite{janson2000random}) and a union bound over all vertices.
Also by Lemma~\ref{lem: properties of gnp}\ref{itm:commonneighbors}, for every vertex $v\in V(G)-\{r\}$, the degree of $v$ into $Q$ is at most $\sqrt{np}$.


\paragraph{\textbf{Exposing $G_3$ to prove niceness}} 
In this part of the proof, we will use the remaining randomness to prove that the collection of paths $\{P_i+\{r_i,r\}\}_{i\in [\delta]}$ is nice, which by \Cref{lem:nice trees are sufficient} is enough to complete the proof.
Consider a vertex $v\in V(G)-\{r\}$. If it is not  in $S$ already, expose its $G_3$-neighbours in $R_v:=V(G)-P-Q-S-N_{G_1}(v)$. Since $|R_v|\geq n-2\cdot 10^5\log n/p-n^{0.2}-4np$, we get by Chernoff bounds that with probability $1-o(n^{-1})$ for every such $v$ its $G_3$-neighbourhood in $R_v$ is of size at least 
\begin{align*}
np_3-\sqrt{4.1np_3\log n}-\max\Big\{10\Big(\frac{2\cdot 10^5\log n}{p}+n^{0.2}+4np\Big)p_3,100\log n\Big\}&\\\geq np_3-\sqrt{4.2np_3\log n}&,
\end{align*}
where we used $\frac{C\log n}{n}<p<\frac{\log ^2n}{\sqrt{n}}$, and the $100\log n$ term ensures that the Chernoff bound gives the required error probability even in cases where the other expression inside the maximum is small. Thus the combined degree of $v$ in $G_1\cup G_3$ into $R_v$ is at least
\begin{align*}
np_1-0.9\sqrt{2np_1\log n}-(10^6\log n + 2np^2+ 2\sqrt{np})& \\
+np_3-\sqrt{4.2np_3\log n}\geq np-0.99\sqrt{2np\log n}&>\delta (G),
\end{align*} 
where we used that the degree of $v$ into $P\cup Q $ is at most $10^6\log n+2np^2+\sqrt{np}$ and the degree of $v$ into $S$ is at most $\sqrt{np}$.

For each $v\notin S\cup \{r\}$, denote $B_v=R_v\cap N_{G_1\cup G_3}(v)$; hence for such $v$ we have $|B_v|\geq \delta$; secondly, for each $v\in S-\{r\}$, denote by $B_v:=N_G(v)\setminus N_G(r)$. Thus, in the first case $|B_v|\geq \delta\geq |I(v)|$
(recall the definition of $I(v)$ from \Cref{def:nice collection}), and in the second one clearly $|B_v|\geq
\delta-|N_G(r)\cap N_G(v)|\geq 
|I(v)|$, as every vertex in $N_G(r)$ already belongs to a distinct tree in our collection.


Fix any $v\in V(G)-\{r\}$. If $v\notin S$, denote $I=[\delta]$; if $v\in S$, let $I\subseteq[\delta]$ be the set of indices $i$ for which $N_G(v)$ does not contain $r_i$.
Now consider the auxiliary bipartite graph $H_v$ with one part $A=\{P_i\}_{i\in I}$ and the other part $B_v$; recall again that $|B_v|\geq |A|$. 
Note also that by \Cref{lem: properties of gnp}\ref{itm:commonneighbors} we have $|A|\geq \delta-\sqrt{np}\sim\delta$.
There is an edge between $P_i\in A$ and $u\in B_v$ in $H_v$ if there is any edge between $P_i$ and $u$ in the graph $G_3$. For each such pair, we expose the edges between $P$ and $B_v$ in $G_3$.

We want to find a matching in $H_v$ which covers $A$, and then we would be done. Indeed, the collection $\{P_i\cup\{r_i,r\}\}_{i\in[\delta]}$ would satisfy \Cref{def:nice collection}, as we would have the appropriate paths of length two for each $i\in I(v)$. Consider a set $B\subseteq B_v$ of size $n_0=|B|=|A|\sim\delta\geq (1-o_C(1))np$, and let us argue that with probability $1-o(n^{-2})$ there is a perfect matching in $H_v[A,B]$; then we would be done by a union bound over all $n$ choices of $v$.

 There is \emph{no edge} in our auxiliary graph between $P_i$ and $w$ with probability $$1-p_0:=(1-p_3)^{|P_i|}\leq 1-p_3|P_i|/2\leq 1-50\log n/np.$$ Thus each edge in $H_v$ is there with probability at least $p_0\geq 50\log n/np\geq 40 \log n/n_0 $ independently. Hence we can invoke \Cref{lem:matching} to get a perfect matching in $H_v[A,B]$ with probability at least 
\[
1-n_0^{n_0p_0/(2\log n_0)-4}\geq 1-e^{40 \log n/2-4\log n_0}\geq 1-o(n^{-2})
\]
with room to spare. This completes the proof, as each one of the steps we performed holds with probability at least $1-o(n^{-1})$, as required for a union bound over all choices of roots $r$.

\end{proof}

\section{Pseudorandom graphs: Proof of \Cref{thm:ndlambda result}}\label{sec:ndlambda}
In this section we prove our asymptotically optimal result for pseudorandom graphs, \Cref{thm:ndlambda result}.
Before we start with the main part of the proof, we will need to introduce some machinery related to embedding trees into expander graphs. 

\subsection{Trees in expanding graphs}\label{sec:FP}

Here we describe the tree‐embedding method based on work of Friedman and Pippenger, often called the extendability approach. Drawing on the foundational embedding theorems of Friedman--Pippenger \cite{friedman1987expanding} and Haxell \cite{haxell2001tree}, it allows us to build trees of size linear in the number of vertices inside an expander graph. Our main lever is \Cref{thm:FP}, which ensures the existence of such embeddings, and we augment it with the straightforward \Cref{lemma:delete}, which shows how to prune leaves from the tree without breaking its extendability. Maintaining extendability in this way is what allows us to apply \Cref{thm:FP} repeatedly. This extendability framework has been the key ingredient in settling many classic open problems in graph theory (see, e.g., \cite{draganic2022rolling,montgomery2019spanning,draganic2024hamiltonicity}). Here we follow the presentation of Montgomery \cite{montgomery2019spanning}; further details are given below. 

We use the following connectivity condition.
\begin{defn}
A graph $G$ is $m$-\emph{joined} for $m>0$, if for every two disjoint vertex subsets $A,B\subseteq V(G)$ of size at least $m$ we have that $e(A,B)>0$.
\end{defn}

We also need the notion of an $(m,D)$-extendable embedding from \cite{montgomery2019spanning}.
\begin{defn}\label{deF:goodness}
Let $m, D \in \mathbb{N}$ satisfy $D \geq 3$ and $m \geq 1$, let $G$ be a graph, and let $S \subset G$ be a subgraph of $G$. We say that $S$ is $(m, D)$-\emph{extendable} if $S$ has maximum degree at most $D$ and
\begin{equation}\label{eq:extendable}
|\Gamma_G(U) \setminus V(S)| \geq (D - 1)|U| - \sum_{x \in U \cap V(S)} (d_S(x) - 1)
\end{equation}
for all sets $U \subset V(G)$ with $|U| \leq 2m$.
\end{defn}

The next theorem is the key technical tool in the section. It allows us to extend an $(m,D)$-extendable embedding of a graph by attaching to one of its vertices a tree of certain size and maximum degree at most $D/2$.
\begin{thm}[Corollary 3.7  in \cite{montgomery2019spanning}]\label{thm:FP}
Let $m, D \in \mathbb{N}$ satisfy $D \geq 3$ and $m \geq 1$, and let $T$ be a tree with maximum degree at most $D/2$, which contains the vertex $t \in V(T)$. Let $G$ be an $m$-joined graph and suppose $R$ is an $(m, D)$-extendable subgraph of $G$ with maximum degree $D/2$. Let $v \in V(R)$ and suppose $|R| + |T| \leq |G| - 2Dm - 3m$. Then, there is a copy $S$ of $T$ in $G - (V(R) \setminus \{v\})$, in which $t$ is copied to $v$, such that $R \cup S$ is $(m, D)$-extendable in $G$.
\end{thm}

The following result shows that adding an edge between two vertices of an $(m,D)$-extendable graph maintains its extendability.
\begin{lem}[Lemma 3.9 in \cite{montgomery2019spanning}] \label{lemma:non-leaf edge}
Let $m, D \in \mathbb{N}$ satisfy $D \geq 3$ and $m \geq 1$, let $G$ be a graph, and let $S$ be an $(m, D)$-extendable subgraph of $G$. If $s, t \in V(S)$ with $d_S(s), d_S(t) \leq D - 1$ and $st \in E(G)$, then $S + st$ is $(m, D)$-extendable in $G$.
\end{lem}

Finally, the last result we need is another simple corollary of the definition of $(m,D)$-extendability which allows for leaf removal.

\begin{lem}[Lemma 3.8 in \cite{montgomery2019spanning}] \label{lemma:delete}
Let $m, D \in \mathbb{N}$ satisfy $D \geq 3$ and $m \geq 1$, let $G$ be a graph, and let $S$ be a subgraph of $G$. Furthermore, suppose there exist vertices $s \in V(S)$ and $y \in N_G(S)$ so that the graph $S + ys$ is $(m, D)$-extendable. Then $S$ is $(m, D)$-extendable. \hfill $\square$
\end{lem}

\begin{rem}
Let us remark here that \Cref{thm:FP}, \Cref{lemma:non-leaf edge} and \Cref{lemma:delete} are not algorithmic in the form they are stated. But, in \cite{draganic2022rolling} an algorithmic version is presented for $(n,d,\lambda)$-graphs. Hence, every step of the proof which uses these three results can also be made into a polynomial time algorithm. More precisely, removing and deleting edges while preserving the property of $(m,D)$-extendability can be implemented in polynomial time. We will also use the Lovasz Local Lemma in a few instances to prove the existence of certain sets --- these too can be made into a deterministic polynomial time algorithm~\cite{moser2010constructive}.
\end{rem}

\subsection{Disjoint paths with prescribed vertex-subsets in $(n,d,\lambda)$-graphs}\label{sec:ndlambda preparation}
We'll need the Lovasz Local Lemma (see Lemma 5.1.1 in \cite{alon2016probabilistic})
\begin{lem}[Asymmetric Lovasz Local Lemma]\label{Lem:ALLL}
Let $\mathcal{A} = \{ A_1, \ldots, A_n \}$ be a finite set of events in the probability space. For $ A \in \mathcal{A} $ let $ \Gamma(A)$ denote the neighbours of $A$ in the dependency graph (In the dependency graph, event $A$ is not adjacent to events which are mutually independent). If there exists an assignment of reals $x : \mathcal{A} \to [0,1) $ to the events such that

$$ \forall A \in \mathcal{A} : \Pr(A) \leq x_A \prod_{B \in \Gamma(A)} (1-x_B) $$
then the probability of avoiding all events in $ \mathcal{A} $ is positive.
\end{lem}

We will use the following version of the Expander Mixing Lemma  for \((n,d,\lambda)\)-graphs (see for example, Corollary 9.2.5 in~\cite{alon2016probabilistic}) which asserts that the number of edges between any two vertex sets asymptotically matches the expected count in a random graph.

\begin{lem}[Expander Mixing Lemma]\label{lem:EML}
Let \(G\) be an $(n,d,\lambda)$-graph and let \(A,B\subseteq V(G)\). Then
$$
\left|e(A,B)-\frac{|A|\,|B|\,d}{n}\right|
\; \le \; \lambda\sqrt{ |A|\,|B|}.\
$$

\end{lem}

Here $e(A,B)$ represents the count of ordered pairs $(a,b)$ where $a \in A$, $b \in B$, and $ab$ forms an edge in graph $G$. It is worth noting that $A$ and $B$ can overlap; specifically, $e(A,A)$ is twice the number of edges within any subset $A \subseteq V(G)$.
The next statement is a corollary of the Expander Mixing Lemma. An elementary proof can be found in~\cite{draganic2025disjoint}.

\begin{lem}\label{lem:expansion}
Let $G$ be an $(n,d,\lambda)$-graph. Then the following hold:
\begin{enumerate}[label=(\alph*)]
\item\label{p:joint} For each two disjoint sets of vertices $X,Y$ of size $|X|,|Y|>\frac{\lambda n}{d}$ it holds that $e(X,Y)>0$.
\item\label{p:expansion} Fix $0<\varepsilon\leq 1$. Let $B,Z\subseteq V(G)$ be such that every vertex in $B$ has at least $\varepsilon d/3$ neighbours in $Z$. Then for every $k$ with $1<k< (\frac{d}{12\lambda})^2$, and for every subset $X\subseteq B$ of size at most $\frac{n}{12k}$, it holds that $|\Gamma_Z(X)|\geq \varepsilon^2k|X|$.
\end{enumerate}
\end{lem}

The following result, the goal of this subsection, states that we can connect subsets of vertices in a pseudorandom graph through a subset of vertices which contains a fraction of the neighbourhood of every vertex. The proof closely follows the one in~\cite{draganic2025disjoint}.

\begin{lem}\label{lem:connecting S_i}
    Let $0<\varepsilon<10^{-3}$, and let $C\geq C_0(\varepsilon)$. Let $G$ be an $(n,d,\lambda)$-graph with $d/\lambda\geq C$.  Let $R,S_1,\ldots S_t$ be pairwise disjoint vertex-subsets of size at least two, such that each vertex in $S_1\cup\ldots\cup S_t\cup R$ has at least $\varepsilon d$ neighbours in $R$. If $\sum_{i=1}^t |S_i|\log(n/|S_i|)\leq \varepsilon  n/100$, then there exist $t$ vertex-disjoint paths $P_i$, where $P_i$ contains $S_i$, and $V(P_i)\subseteq S_i\cup R$.
\end{lem}

\begin{proof}
    By \Cref{lem:expansion} applied with $B:=S_1\cup\ldots\cup S_t\cup R$, and $Z:=R$ and $k=1/\varepsilon^3$, we get that $|\Gamma_R(X)|\geq |X|/\varepsilon$ for every $X\subseteq B$ of size at most $|X|\leq \varepsilon^3n/12.$
    
    Consider the graph $G'=G[R\cup S_1\ldots\cup S_t]$, and observe that $|G'|\geq \varepsilon n/2$ by the Expander Mixing Lemma, as
    \[
    |G'|\varepsilon d\leq e(V(G'),V(G'))\leq |G'|^2\frac{d}{n}+\lambda |G'|,
    \]
which together with $\lambda<\varepsilon d/2$ implies the claim.
    Note that $G'\subset G$ is $m$-joined for $$m=\varepsilon^5n\geq \lambda n/d+1$$ by \Cref{lem:expansion}. Furthermore, by the above we have that the independent set $I:=I[S_1\cup\ldots\cup S_t]$ is an $(m,10)$-extendable subgraph of  $G'$. Indeed, by the first paragraph of the proof
    \begin{equation*}
|\Gamma_G(U) \setminus V(I)|=|\Gamma_R(U)| \geq |U|/\varepsilon\geq10|U|  - \sum_{x \in U \cap V(I)} (d_I(x) - 1),
\end{equation*}
for all $|U|\leq 2m$, since $d_I(x)\geq 0$ for all $x\in V(I)$ (in fact, $d_I(x)=0)$.
Hence, we can use \Cref{thm:FP}, \Cref{lemma:delete}, and \Cref{lemma:non-leaf edge} to attach and remove leaves from $V(I)$ in $G'$, as long as the total size of the current forest is at most $$s:=\varepsilon n/4\leq |G'|-20m-3m.$$
We connect the sets $S_i$ one by one through $R$ using the given machinery, by always adding a path between two vertices in $S_i$.
We will now show how to construct the required path for $S_1$, and later we will argue how to do the same for the remaining $S_i$'s analogously. 

 First, we extend the $(m,10)$-extendable subgraph $I$ in $G'$ by attaching  disjoint binary trees of size $\frac{s}{3|S_1|}$ (and of depth at most $\log_2 \frac{s}{3|S_1|}$) to each vertex in $S_1$. Note that this can be done using~\Cref{thm:FP}, since the total size of the constructed forest (including all vertices in $S:=S_1\cup\ldots\cup S_t$ whose number is at most $\varepsilon n/100$) is at most $s/3+|S|\leq s/2$. Now, consider the set $L_1$ of the vertices of an arbitrary choice of $\lfloor|S_1|/2\rfloor$ of the trees, and $L_2$ the vertices of the remaining trees. Since both sets contain at least $s/9> \lambda n/d$ vertices each, there is an edge between the two sets by~\Cref{lem:expansion}\ref{p:joint}. Adding this edge to the current $(m,10)$-extendable graph creates a path $P$ of length at most $2\log_2 \frac{s}{3|S_1|}+1$ between two vertices in $S_1$, while the constructed graph stays $(m,10)$-extendable by \Cref{lemma:non-leaf edge}. Now we can use Lemma~\ref{lemma:delete} to remove all vertices (leaf-by-leaf) of the attached trees, except those in the path $P$, so that we are again left with an $(m,10)$-extendable subgraph.

We now repeat this process, until we have connected all vertices in $S_1$. More precisely, we perform the following inductively described procedure. Suppose that after the $i$-th step, for $0<i<|S_1|-1$, we have constructed internally vertex-disjoint paths $P_1,P_2,\ldots, P_i$ with endpoints in $S_1$, that satisfy the following invariants:
\begin{itemize}
\item the internal vertices are in $R$; 
\item the length of $P_j$ for all $j\leq i$ is at most $2\log_2 \left(\frac{n}{|S_1|-j+1}\right)+1$;
\item The graph $I[S]\cup P_1\cup\ldots \cup P_i$ is a linear forest and an $(m,10)$-extendable subgraph of $G'$;
\item the number of components (i.e., paths) in $I[S_1]\cup P_1\cup\ldots \cup P_i$ is $|S_1|-i$.
\end{itemize}
In the $(i+1)$-th step we want to construct a path $P_{i+1}$ that maintains the invariants. 
In the end, after the $(|S_1|-1)$-th step, we will then have the required path which contains all vertices in $S_1$.

For each path in the linear forest $I[S_1]\cup P_1\cup\ldots \cup P_i$, choose one endpoint, and let $S_1'$ denote the set of these endpoints, noting that $|S_1'|=|S_1|-i$. As before, extend the $(m,10)$-extendable subgraph $I[S]\cup P_1\cup\ldots \cup P_i$ by attaching a binary tree to each vertex in $S_1'$, each one of size $\frac{s}{3|S_1'|}$, and hence of depth at most $\log_2\left(\frac{s}{|S_1'|}\right)\leq\log_2  \left(\frac{n}{|S_1|-i}\right)$. We can indeed use~\Cref{thm:FP}, as the current forest consists of $S$ and the now constructed trees (which together contain at most $s/2$ vertices), and the previously found paths $P_1,\ldots, P_i$, whose total number of vertices is at most: 

\begin{align}
\sum_{1\leq j\leq i}\left(2\log_{2} \left(\frac{n}{|S_1|-j+1}\right)+1\right) 
&\leq |S_1|+2\sum_{j=1}^{|S_1|-2}\log_{2} \left(\frac{n}{|S_1|-j+1}\right)\notag \\
&= |S_1|+2\sum_{j=3}^{|S_1|}\log_{2} \left(\frac{n}{j}\right)\notag \\
&\leq|S_1|+2|S_1|\log_{2} n-2\sum_{j=3}^{|S_1|}\log_{2} j\notag \\
&=|S_1|+2|S_1|\log_{2} n-2\log_{2} \left((|S_1|)!/2\right)\notag \\
&\leq |S_1|+2|S_1|\log_{2} n-2|S_1|\log_{2} \left(\frac{|S_1|}{e}\right)\notag \\
&\leq |S_1|(1+2\log_2\Big(\frac{ne}{|S_1|}\Big))\leq 3|S_1|\log_2\Big(\frac{n}{|S_1|}\Big)\leq 3\varepsilon n/100.\label{total number in paths}
\end{align}

Hence, in total, the forest at any point contains less than $s$ vertices, so we can successfully apply~\Cref{thm:FP}.

As $i<|S_1|-1$, by the induction hypothesis we have at least $|S_1|-i\geq 2$ trees, so we can group them into two collections each containing at least a third of all vertices. Since the total number of vertices in all trees is $s/3$, each of the two collections contains at least $\frac{s/3}{3} > m$ vertices. Thus, there must be an edge between them.
 This edge closes a path $P_{i+1}$ connecting two endpoints in the linear forest $I[S_1]\cup P_1\cup\ldots \cup P_i$. Adding this edge via \Cref{lemma:non-leaf edge} and removing the unused edges (leaf-by-leaf) using~\Cref{lemma:delete}, we are left with an $(m,10)$-extendable subgraph 
$I[S_1]\cup P_1\cup\ldots \cup P_{i+1}$ of $G'$.
Let us verify that all invariants are still satisfied. 

By construction, all internal vertices of the new path are in $R$.
The length of the new path is at most twice the depth of the trees plus one, i.e. it is of length at most $2\log_{2}  \left(\frac{n}{|S_1|-i}\right)+1$.
By construction, the subgraph $I[S_1]\cup P_1\cup\ldots \cup P_i\cup P_{i+1}$ is still $(m,10)$-extendable. Furthermore, as $P_{i+1}$ connects two endpoints of a linear forest through a path which is internally vertex-disjoint from $I[S_1]\cup P_1\cup\ldots \cup P_i$, we obtain a new linear forest with one less component than before, resulting in $|S_1|-i-1$ components/paths, as required. 

Continuing from the final $(m,10)$-extendable embedding (which now connects all vertices in $S_1)
$, we can now follow the same procedure for the sets $S_2,\ldots, S_{t}$, one by one, always extending the current $(m,10)$-extendable subgraph. Here we stress that when we are about to do the procedure for the set $S_i$, we continue extending the $(m,10)$-extendable subgraph left after having completed treating $S_{i-1}$, repeating the exact same argument as before, only bearing in mind that we have already used up some vertices in $R$ for the previous steps of the embedding.

The only thing we need to verify is that the total number of vertices in the forest which we construct is always at most $s$. But this is straightforward to see, as the total number of vertices in the current binary trees which we use in every step is at most $s/2$ by construction. On the other hand, the total number of vertices in paths used for all $S_i$, by (\ref{total number in paths}), is at most $$3\sum_{i=1}^t|S_i|\log_2\Big(\frac{n}{|S_i|}\Big)\leq 3\varepsilon n/100<s/2,$$ so we are done.

\end{proof}


\subsection{ISTs in pseudorandom graphs}

We now prove our second main result, restated here for the reader’s convenience.
\NDLambdaResult*
\begin{proof}
    We will prove the claim with $10\varepsilon$ instead of $\varepsilon$; we may also assume that $\varepsilon<10^{-3}$ since the lemma only gets stronger by decreasing $\epsilon$.
    Define $k=(1-10\varepsilon)d$.
    Fix an arbitrary root $r$, and $k$ of its neighbours $L:=\{v_1,\ldots,v_k\}$. 
The proof follows from the next two claims. 

\begin{claim}\label{cl:good partition}
There exists a partition of $V(G)\setminus(\{r\}\cup L):=U\cup R\cup S_1\cup\ldots\cup S_k$ such that the following holds for every $v\in V(G)$:     
\begin{itemize}
    \item Let $I(v)$ be the set of $i\in[k]$ such that $v$ is not a neighbour of $v_i$. Then for each $i\in I(v)$ there are edges $u_iy_i$, where $u_i\in N_U(v)$ are distinct vertices, and $y_i\in S_i$;
    \item $|N_{R}(v)|\geq \varepsilon^3 d$;
    \item $|S_i|\leq \varepsilon^{10}n/(d\log d)$.
\end{itemize}
\end{claim}

\begin{claim}\label{cl:connecting paths}
    Given a partition as in \Cref{cl:good partition}, there exist disjoint paths $P_1\ldots,P_k$ such that $$S_i\cup \{v_i\}\subseteq V(P_i)\subseteq S_i\cup R \cup \{v_i\}.$$

\end{claim}
Before we prove the two claims, we observe that they suffice to prove the theorem. 
Define $T_i:=P_i+(r,v_i)$.
Since in \Cref{cl:good partition} we have $u_i\in U$, and $U\subset V(G)\setminus \bigcup_{i\in[k]}T_i$, using the first bullet point of Claim~\ref{cl:good partition} we have that $\{T_i\}_{i\in [k]}$ is a nice collection of trees, so we are done by \Cref{lem:nice trees are sufficient}.

\begin{proof}[Proof of \Cref{cl:good partition}]

Consider an arbitrary vertex $v\neq r$. Let $t_v:=k-|N_L(v)|$, and denote $B_v:=N(v)\setminus N_L(v)\cup\{r\}$, noting that $|B_v|=  t_v+10\varepsilon d$. At most $\varepsilon d$ vertices in $B_v$ have at least $\varepsilon d$ neighbours in $B_v\cup L$ as by the Expander Mixing Lemma we have

\[
e(B_v,B_v\cup L\cup \{r\})\leq d\cdot2d\cdot d/n+\lambda \sqrt{d\cdot 2d}\leq \varepsilon^2 d^2
\]
where we used $d=o(n)$ and $\lambda=o(d)$.
Thus, we can take a subset $B_v'\subseteq B_v$ with $|B_v'|= t_v+\varepsilon d$, of vertices with at least $\varepsilon d$ neighbours outside $W:=B_v\cup L\cup\{r\}$. 
Thus every subset $X\subseteq B_v'$ has $|N_{G-W}(X)|\geq C\log ^2d|X|$ by \Cref{lem:expansion} by setting $k_{4.7}=C\log^2 d$, $Z_{4.7}=V(G)\setminus W$, for $C>C_0(\varepsilon)$ large enough so that the rest of the argument goes through.
Hence, by the generalised Hall's theorem we can find a collection of $t_v+\varepsilon d$ vertex disjoint stars $Q_i$ of size $C\log^2 d$ each, where $i\in [t_v+\varepsilon d]$, with centers in $B_v'$. The need for size $C\log^2 d$ is the bottleneck that dictates the assumed bound on $d/\lambda$.

    Now we construct the required partition. For each vertex $v$ in $V(G)-L-\{r\}$, randomly assign it to one of the following sets, with given probability:
    \begin{itemize}
        \item with probability $\varepsilon^{11}/(d\log d)$ put $v$ into $S_i$, for each $i\in [k]$;
        \item with probability $\varepsilon/100$ put $v$ into the \emph{reservoir} $R$;
        \item with the remaining probability, put $v$ into $U$, the set of unassigned vertices.
    \end{itemize}
\noindent 
For each $v\in V(G)$, denote by $A_v$ the intersection of the following events:
\begin{enumerate}[label=(\arabic*)]
    \item There are at least $t_v+\varepsilon d/100$ vertices in $B_v'\cap U$. By possibly relabelling, let $Q_1,\ldots,Q_{t_v+\varepsilon d/100}$ be their previously chosen stars.
    \item There is an injection $\sigma:I(v)\rightarrow [t_v+\varepsilon d/100]$, such that $Q_{\sigma(i)}$ contains a leaf in $S_i$, for all $i\in I(v)$.
    \item $|N_{R}(v)|\geq \varepsilon^3 d$.
\end{enumerate}
Note that if $\bigcap_v A_v$ holds, then the first two bullets of \Cref{cl:good partition} hold (with $u_i$ the center of $Q_{\sigma(i)}$ and $y_i$ the leaf which guaranteed in (2)).
Let us estimate the probability $P[A_v]$. First, note that the first property is satisfied with probability at least $1-e^{-\Theta(d)}$, as $\mathbb E[|B_v'\cap U|]\geq |B_v'|(1-\varepsilon/3)\geq (1-\varepsilon/3)t_v+\varepsilon d/2\geq t_v+\varepsilon d/6$, so a Chernoff bound applies.

Conditioning on the first property, for the second one consider the auxiliary bipartite graph $H_v$ with first part $I(v)$ of size $t_v$, and second part $[t_v+\varepsilon d/100]$, where $(i,j)$ is an edge if there is a vertex $w\in S_i\cap Q_j$. To show that $H_v$ has a matching covering $I(v)$ with sufficiently high probability, set $A := I(v)$ and $B := [t_v + \varepsilon d / 100]$. We claim the following.

\begin{prop}\label{prop: matching in pseduorandom}
With probability at least $1 - d^{-6}$ there is a matching between $A$ and $B$ that saturates $A$.
\end{prop}
\begin{proof}
    If $H_v[A,B]$ does not have a required matching matching, then  by Hall's theorem there is a set $A_0\subset A$ of cardinality $|A_0|\le t_v$ with $|N(A_0,B)|< |A_0|$. 
    For $A_0$ to have no neighbours in $B_0\subseteq B$, it is necessary that, for every $i \in B_0$, all leaves of the stars $Q_i$ lie outside the sets $S_j$ with $j \in A_0$. The probability of this event, for fixed sets $A_0$ and $B_0$, where $|B_0|=|B|-|A_0|$, is
    \[
     \left(1-\frac{|A_0|\varepsilon^{11}}{d\log d}\right)^{|B_0|C\log^2 d }\leq e^{-\frac{|A_0|\varepsilon^{11}}{d\log d}(t_v+\varepsilon d/100-|A_0|)C\log^2 d}\leq d^{-\frac{\varepsilon^{12}C|A_0|}{100}} 
    \]
    where we used that $|A_0|\leq t_v$.
    Thus, by a union bound, the probability that 
    Hall's condition for side $A$ is violated in $H_v$ is at most
    \begin{align*}
    \sum_{\ell=1}^{t_v} \binom{t_v}{\ell}\binom{t_v+\varepsilon d}{\ell}d^{-\frac{\varepsilon^{12}C\ell}{100}}\le \sum_{\ell=1}^{t_v} \left(\frac{e(t_v+\varepsilon d)}{\ell}\right)^{2\ell}d^{-\frac{\varepsilon^{12}C\ell}{100}}\\\leq \sum_{\ell=1}^{t_v} \left(3d\right)^{2\ell}d^{-100 \ell}\leq d^{-6},
    \end{align*}
where we used that $d>t_v$. This completes the proof of the \Cref{prop: matching in pseduorandom}.
\end{proof}

Finally, the third property also holds with probability at least $1-e^{-\Theta(d)}$, as the expected number of neighbours in $R$ is at least $\varepsilon d\cdot \varepsilon/100$. 
Hence we conclude that $P[A_v]\geq 1-3d^{-6}$. Note that each event $A_v$ depends on at most $d^{4}$ other events, since this bounds the number of vertices within distance $4$ of $v$; for vertices $w$ at distance at least $5$, the events $A_w$ are independent of $A_v$.

Let $E_i$ be the event given by the third bullet of \Cref{cl:good partition}, i.e. the event that $|S_i|\leq \varepsilon^{10}n/(d\log d)$. By Chernoff bounds, we have $P[E_i]\geq 1-e^{-\Theta(n/(d\log d))}\geq 1-n^{-3}$, as by assumption $d=o(n/\log^2n)$. Note that each $E_i$ might depend on each event $A_v$ (the number of which is $n$) and each other $E_j$ (of which we have at most $d$).

Thus, by the Asymmetric Local Lemma (\Cref{Lem:ALLL}), there is a choice of sets $R,S_1,\ldots , S_k$ such that each $A_v$ and each $E_i$ hold. Indeed, by setting $x_v=d^{-5}$ for all $v\neq r$, and $x_i=n^{-2}$ for all $i\in[k]$, the condition of the lemma is satisfied as:

\[
P\Big[\overline{A_v}\Big]\leq 3d^{-6}\leq  d^{-5}(1-d^{-5})^{d^4}\cdot (1-n^{-2})^d\sim d^{-5},
\]
and:
\[
P\Big[\overline{E_i}\Big]=e^{-\Theta(n/(d\log d))}\leq n^{-2}(1-d^{-5})^n(1-n^{-2})^d\sim e^{-\Theta(n/d^5)}.
\]
This completes the proof.
\end{proof}

Now we turn to the second claim.
\begin{proof}[Proof of \Cref{cl:connecting paths}]
Denote $S_i':=S_i+v_i$.
We apply \Cref{lem:connecting S_i},  to the sets $R , S_1',\ldots, S_k'$. 
%
Indeed, we can do that since each term in $\sum |S_i'|\log (n/|S_i'|)$ is maximized when $|S_i'|$ is as large as possible; 
by \Cref{cl:good partition} we can assume that $|S_i|\leq \varepsilon^{10} n/(d\log d)$, thus we have 

\[\sum_{i=1}^k |S_i'|\log (n/|S_i'|)\le d\cdot  \frac{\varepsilon^{10}n}{d\log d}2\log d=2\varepsilon^{10} n.
\]
This connects the set $S_i'=S_i+v_i$ via path $P_i$ through $R$, such that the $P_i$'s are pairwise vertex-disjoint. This completes the proof of the second claim.
\end{proof}

As we have proven the two required claims, this completes the proof of \Cref{thm:ndlambda result}   
\end{proof}

\section{Concluding Remarks}
 We have proven the Zehavi-Itai IST conjecture for random graphs \( G(n, p) \) for essentially the whole range of relevant values of the edge probability $p(n)$, specifically for  \(C\log n / n \leq p(n)\leq 0.99\), for a large enough absolute constant $C>0$.  We have also shown that the conjecture holds asymptotically (namely, one can find $(1-o(1))\kappa(G)$ ISTs) for a broad class of weakly pseudorandom graphs (specifically, \((n,d,\lambda)\)-graphs, under a rather mild assumption on the eigenvalue ratio $d/\lambda= \omega(\log d)$). As a consequence, the conjecture also holds asymptotically for random \(d\)-regular graphs when \(d\) is at least a large constant. In other words, \(d\)-regular graphs typically contain a family of  \((1 - o(1))d\) ISTs.

Several problems remain open. For general \(k\)-connected graphs, the best known lower bound~\cite{censor2017tight} guarantees only \(\Omega(k / \log^2 n)\) independent spanning trees, and improving this remains an exciting challenge. Our techniques offer some insight in this direction, and it would be worthwhile to investigate whether one can exploit a dichotomy between good expansion and small cuts to obtain stronger lower bounds.

The conjecture is known to hold for small values of \(k\), specifically for \(k \leq 4\). It would be interesting to explore the other end of the spectrum, namely  graphs with linear connectivity, and to determine whether one can prove the conjecture in this setting or at least to improve significantly  the lower bound from~\cite{censor2017tight}.

Another immediate challenge is to strengthen \Cref{thm:ndlambda result} by replacing the hypothesis $d/\lambda=\omega(\log d)$ with $d/\lambda \ge C(\varepsilon)$ for a sufficiently large constant $C(\varepsilon)$ depending only on $\varepsilon$. We can notice that this strengthening holds under the additional assumption that $G$ is $K_{s,t}$-free for fixed integers $s,t\ge2$: a constant spectral gap already yields polynomial expansion in this case. In particular, in our proof one can then find stars of size $d^{c}$ for some $c>0$, whereas the proof requires only stars of order $C\log^{2} d$, which we obtained using the Expander Mixing Lemma. Indeed, combining the Expander Mixing Lemma with a Kővári--Sós--Turán--type argument (see, e.g.,~\cite{liu2017proof}) gives the needed polynomial expansion and hence the required stars.

Finally, it would be interesting to address the remaining range of $p$ for binomial random graphs as well, and in particular at the weak threshold for connectivity, when $\log n/n\leq p=O(\log n/n)$. 

Although our result is stated existentially, our proof yields a polynomial-time algorithm, as can be verified by examining each step of the argument.

\vspace{1cm}

\noindent\textbf{Acknowledgement.} Part of this research was conducted at Princeton University during the workshop on Graph expansion and its applications in April-May 2025. We would like to thank Matija Buci\'c for organizing the workshop, and Princeton University for great hospitality.

\bibliographystyle{abbrv}

\end{document}